\documentclass[10pt]{article}
\usepackage{amsmath}
\usepackage{amsthm}
\usepackage{amsfonts}
\usepackage{amssymb}
\usepackage{latexsym} 

\usepackage[matrix,tips,graph,curve]{xy}



\makeatletter 
\@addtoreset{equation}{section}
\makeatother

\theoremstyle{plain}
\newtheorem{theorem}[equation]{Theorem}

\newtheorem{proposition}[equation]{Proposition}

\theoremstyle{definition}
\newtheorem{definition}[equation]{Definition}

\newcommand{\IC}{\mathbb{C}}

\newcommand{\IQ}{\mathbb{Q}}

\newcommand{\IZ}{\mathbb{Z}}

\newcommand{\tr}{\mathrm{tr}}

\renewcommand{\deg}{\mathrm{deg}}

\def\d/{/\mspace{-6.0mu}/}

\newcommand{\db}{\bar{\partial}}

\newcommand{\p}{\partial}

\begin{document}

\title{$C^{\infty}$ Stability, Canonical Maps, and Discrete Dynamics}
\author{Mark Stern}

\footnotetext{Duke University, Department of Mathematics;  
e-mail:  stern@math.duke.edu,\\ partially supported by NSF grant DMS 1005761}
\date{}

\maketitle
\setcounter{section}{0}
\section{Introduction}
A fundamental question in Hodge theory is: 'Which complex vector bundles on a smooth projective variety admit a holomorphic structure?' If $E$ is a holomorphic vector bundle on a projective variety $M$, then the degree $2p$ component of the chern character of $E$, which we denote $ch_p(E)$, lies in $F^pH^{2p}(M,\IC)\cap \bar F^pH^{2p}(M,\IC)$, where $F^{\cdot}$ denotes the Hodge filtration. This necessary condition on the existence of a holomorphic structure is both nonlinear and nonlocal, making it extremely difficult to see how to utilize this hypothesis in analytic approaches to constructing holomorphic structures. In this note we ask an easier converse question: If $ch_p(E)\not\in F^{p-j}H^{2p}(M,\IQ)\cap \bar F^{p-j}H^{2p}(M,\IQ),$ how is $j$ - the distance to the $(p,p)$ axis in the Hodge diamond - reflected in singularity formation  in analytic attempts to construct (the nonexistent) holomorphic structures? In order to describe our results, we first introduce one natural scheme for seeking holomorphic structures.  

If $E'$ is a holomorphic bundles, and $L$ is an ample line bundle, then for $k$ sufficiently large, 
$$dim(H^0(E'\otimes L^k)) = \chi(M,E'\otimes L^k),$$
where $\chi$ denotes the holomorphic Euler characteristic. Viewed as the index of the  elliptic operator $\db+\db^*$, the holomorphic Euler characteristic can be extended to arbitrary hermitian complex bundles endowed with a semiconnection. For an arbitrary complex hermitian bundle $E$ on a Kahler manifold $(M,g)$, we define the energy 
$$\Phi^{0,1}(h,H,\db_A) = \sum_a\|\db_A s_a\|^2,$$
where $\db_A$ is a semiconnection on $E$, $H$ is a $\chi(M,E)$ dimensional subspace of $C^\infty(M,E)$, $\{s_a\}_a$ is a unitary $L_2$ basis for $H$, and the $L_2$ norms are defined by $h$ and the Kahler metric. If $E$ admits a holomorphic structure with $dim(H^0(E)) = \chi(M,E)$, then $\Phi^{0,1}(h,H,\db_A)=0$ for an appropriate choice of $H$ and $\db_A$. If, on the other hand, $\Phi^{0,1}$ achieves $0$ for some $H$ whose sections globally generate $E$, then $E$ admits a holomorphic structure. Hence a natural problem is to understand minimizing sequences of $\Phi^{0,1}$. When is $\Phi^{0,1}$ bounded away from zero? If $0$ is the infimum of $\Phi^{0,1}$, is it achieved? In this note, we begin studying the relation between the distance of $ch(E)$ from the $(p,p)$ axis in the Hodge diamond to the formation of singularities in minimizing sequences of $\Phi^{0,1}$. Our treatment naturally leads to the consideration of two different types of singularity formation. The first is simply the loss of regularity of the sections of $H$ with respect to a fixed background connection. The second is really a stability issue. We call $H\subset C^\infty(M,E)$ {\em admissible} if its sections globally generate $E$. Let $\Pi_H$ denote the the Schwartz kernel of the $L_2$ unitary projection onto $H$. The global generation of $E$ by $H$ is equivalent to the invertibilty of $\Pi_H(x,x)^{-1}$. (Evaluate on the diagonal before inverting.) We call a sequence of subspaces $\{H_n\}_n$  {\em stable} if $\Pi_{H_n}(x,x)^{-1}$ is uniformly bounded. 

In the special case of $E$ holomorphic, $L$ ample, Donaldson \cite{D1}  considered the discrete dynamical system of metrics $h_n$ on $E\otimes L^k$ defined by setting $H= H^0(M,E\otimes L^k)$, and setting
$$h_{n+1}(\cdot,\cdot) := \frac{rank(E)vol(M)}{dim(H)}h_{n}(\cdot, \Pi_{H,h_n}(x,x)^{-1}\cdot ),$$
where now $\Pi_{H,h_n}$ denotes the Schwartz kernel of the $L_2$ projection onto $H$ unitary with respect to the metric $h_n$. Fixed points of this dynamical system are called {\em balanced metrics}. Donaldson conjectured that the existence of balanced metrics was equivalent to a stability condition. Wang  \cite{w}  (see also \cite{K}) proved that holomorphic bundles on projective manifolds admit balanced metrics if and only if they are Gieseker polystable.

The metrics, $h_{n+1}$, in Donaldson's dynamical system arise from pulling back the metric on the universal bundle over the Grassmannian $Gr(rank(E\otimes L^k),H)$ by the canonical map $i_{(h_n,H)}$ determined by $(h_n,H)$. We introduce two discrete dynamical systems, $T^{0,1}$ and $T$, which extend Donaldson's metric dynamical system for holomorphic bundles to dynamical systems for $C^\infty$ complex bundles. In the system $T$, the metric, the subspace $H$ of $C^\infty(M,E)$, and the connection vary. The flow $T^{0,1}$ is complementary to Donaldson's in that the metric is fixed and only the subspace $H$ of $C^\infty(M,E)$  and the connection vary. In these new systems, the connection is also pulled back from the Grassmannian by a canonical map. The dynamical system $T^{0,1}$ has the amusing feature of factoring the nonlinear problem of minimizing $\Phi^{0,1}$ into an elementary linear pde problem and  an explicitly (algebraicly) solvable nonlinear problem.

 Using comparison to these dynamical systems, we obtain relations between Hodge structures and rate of singularity formation such as the following.

\begin{theorem}\label{A}
Suppose that $ch_p(E)\not\in F^{p-j+1}H^{2p}(M)\cap \bar F^{p-j+1}H^{2p}$. Let $\{H_n\}_n $ be a stable sequence of subspaces with $L_2$ unitary basis $\{s_{a,n}\}_a$ and $\{A_n\}_n $
 be a sequence of smooth hermitian connections. If $\sum_a|\db_{A_n}s_{a,n}|^2\stackrel{L_{p}}{\rightarrow} 0,$ then 
$\sum_a|d_{B_n}s_{a,n}|^2\stackrel{L_{p}}{\rightarrow} \infty,$ for any sequence of hermitian connections $\{B_n\}_n$. 
If $\sum_a|\db_{A_n}s_{a,n}|^2\stackrel{L_{\infty}}{\rightarrow} 0,$ then 
$\sum_a|d_{B_n}s_{a,n}|^2\stackrel{L_{p-\frac{j}{2}}}{\rightarrow} \infty,$ for any sequence of hermitian connections $\{B_n\}_n$. 
\end{theorem}
We  also have elementary information about lower bounds for $\Phi^{0,1}$.
\begin{theorem}\label{B}
If $c_1(E)\not\in F^1H^{2}\cap \bar F^1H^{2}$, then  $\ln(\Phi^{0,1})$ is bounded below on stable sequences. 
\end{theorem}

In each case, our results reduce to comparison to canonical connections associated with canonical maps to Grassmann manifolds. 
Hence we begin in Section \ref{canmap} with a discussion of canonical maps and canonical connections.  

Minimizers of $\Phi^{0,1}$ and $\Phi$ are fixed points of $T^{0,1}$ and $T$ respectively. We call a semiconnection $\db_A$ {\em balanced} with respect to $(h,H)$ if $(H,\db_A)$ is a fixed point of $T^{0,1}$. Equivalently, $\db_A$ is balanced  with respect to $(h,H)$ if it is the pullback via $i_{(h,H)}$ of the canonical connection on the universal bundle on the Grassmannian. We include a digression in Section \ref{2ndorder} on the relation between canonical maps, balanced connections, and harmonic maps. 

After proving Theorems \ref{A} and \ref{B} in Section \ref{obstructions}, we conclude with several tangentially related observations springing from the study of canonical connections rather than from Hodge theoretic considerations.    For example, let $L:= \nabla_A^*\nabla_A +W$, for $W\in C^\infty(End(E))$. Let $\{\psi_j\}_{j=1}^\infty$ be an $L_2-$unitary eigenbasis $L^2(M,E)$, with $L\psi_j = \lambda_j\psi_j$. Then 
\begin{theorem}\label{entropy}
\begin{equation}\label{entreq}F_{A } =  (4\pi t)^{n/2}\sum_{\lambda_j}e^{-t\lambda_j }d_A\psi_j\langle \cdot,d_A\psi_j\rangle + O(t) .  \end{equation}
\end{theorem}
This equality can be proved by heat equation asymptotics, but also arises naturally from consideration of canonical maps and balanced connections.

\section{Canonical Maps}\label{canmap}
Let $M$ be a compact Riemannian manifold. Let $(E,h,d_A)$ be a rank $k$ Hermitian vector bundle over $M$ with metric $h$ and metric  compatible connection $d_A$.  Let $H$ be an $N-$ dimensional subspace of smooth sections of $E$.  Let $\{s_a\}_{a=1}^N$ be an $L_2$-unitary basis of $H$. Let $\Pi_{H,h}$ denote the Schwartz kernel of the $L_2$ projection onto $H$. We write $\Pi_H$ or $\Pi$ when $h$ or $(H,h)$ are understood.  
\begin{definition}
We call a subspace  $H\subset C^\infty(M,E)$ {\em admissible} if $\Pi_{H}(x,x)$ is invertible for all $x\in M$.
\end{definition} 
$H$ is admissible if and only if $E$ is globally generated by sections of $H$.  Thus $\{s_a(x)\}_{a=1}^N$ spans $E_x$ for all $x\in M$.    Let $\{e_a\}_{a=1}^N$ be a unitary basis for $\IC^N$. Let $\{e_a^*\}_a$ denote the dual basis.  Define a rank $k$ projection operator on $\IC^N$ as follows. 
$$V:= V_{ab}e_b\otimes e_a^*,$$
where 
\begin{equation}\label{vdef}V_{ab}(x) = \langle s_a(x), \Pi^{-1}(x,x)s_b(x)\rangle .\end{equation}
Fix $x$ and choose an orthonormal basis $\{\tilde s_a\}_a$ of $H$ so that $\langle \tilde s_a(x),\tilde s_b(x)\rangle(x) = \lambda_{a}\delta_{ab},$ with $\lambda_a = 0,$ for $a>k$. For this frame, defining $\tilde V$ analogous to $V$, we have 
$$ \tilde V(x) = \sum_{a\leq k}  e_a\otimes   e_a^*.$$
As the two matrices are conjugate under an orthonormal change of frame, we see 
$V$ is a rank $k$ orthogonal projection, and $V$ defines a map 
$$i_{(h,H)}:M\to Gr(k,N).$$ 
We suppress the basis dependence of this map as a change of unitary  basis corresponds to composition with an isometry of $Gr(k,N)$. 

In this notation, the universal bundle over $Gr(k,N)$ is the subbundle of 
$Gr(k,N)\times \IC^N$ of the form $U:= \{(V,w)\in Gr(k,N)\times \IC^N: w=Vw\}.$ This bundle is equipped with a canonical
connection 
$$d_C  s := Vd s,$$
where $d$ is the trivial connection on the ambient trivial bundle.  The canonical connection is compatible with the metric. The universal bundle also has a canonical subspace of sections $H_C$ spanned by 
$$z_a(V) := Ve_a.$$
We may naturally identify $i_{(h,H)}^*z_a$ with $s_a$. The standard metric $\langle\cdot,\cdot\rangle$ on the trivial bundle induces a metric on the canonical bundle. We compute 
$$\langle z_a(x), z_b(x)\rangle(x) = \langle V_{ap}(x)e_p , V_{bq}(x)e_q\rangle =   V_{ab}(x).$$

\begin{definition}We say $h$ is {\em balanced, relative to $H$} if 
$$h = i_{(h,H)}^*\langle\cdot,\cdot\rangle.$$

\end{definition}

Applied to the canonical sections, the connection gives 
$$d_C z_a(V) = VdVe_a.$$
From $V^2 = V,$ we have $V(dV)V = 0$. 
Hence, for each $x\in Gr(k,N)$ and every constant section $e_x$ of the trivial $\IC^N$  bundle satisfying $e_x=V(x)e_x,$ the canonical connection satisfies 
\begin{equation}\label{charC}d_C s(x) = 0, \text{ for } s(y) := V(y)e_x.\end{equation}
Equation (\ref{charC}) characterizes the canonical connection. 

\begin{definition}We say $d_A$ is {\em balanced with respect to $(h,H)$} if $d_As(x) = 0$ when $\langle s,\psi\rangle_{L_2} =0,$ $\forall\psi\in H$ such that $\psi(x) = 0.$
Equivalently, $$d_A = i_{(h,H)}^*d_C.$$
Similarly, if the underlying manifold is complex, we say a semiconnection $\db_A$ is {\em balanced with respect to $(h,H)$} if $\db_As(x) = 0$ when $\langle s,\psi\rangle_{L_2} =0,$ $\forall\psi\in H$ such that $\psi(x) = 0.$
\end{definition}
\section{$H$ defined by second order equations}\label{2ndorder}
In this section, we compute the equations satisfied by $V$, when $H$ is a subspace of the solution space of a first or second order elliptic equation, as happens in the special case when $H$ is generated by holomorphic sections.  
Let $Q$ denote the complementary projection: 
\begin{equation}\label{qdef}Q=I-V.\end{equation}
Then in local coordinates, 
\begin{equation}\label{dV}V_{ab,j}(x) = Q_{ac}\langle s_{c;j} , \Pi^{-1} s_b \rangle 
                                              + \langle s_a , \Pi^{-1} s_{c;j}\rangle Q_{cb}.\end{equation}    
Differentiating again gives 

$$ V_{ab,ji}(x) = Q_{ac,i}\langle s_{c;j} , \Pi^{-1} s_b \rangle +Q_{ac}\langle s_{c;ji} , \Pi^{-1} s_b \rangle+Q_{ac}\langle s_{c;j} , \Pi^{-1} s_{b;i} \rangle Q_{pb}$$
$$ -Q_{ac}\langle s_{c;j} ,  \Pi^{-1}s_{p}  \rangle \langle s_{p;i}, \Pi^{-1}  s_b  \rangle
                                              +  \langle s_a , \Pi^{-1} s_{c;j}\rangle Q_{cb,i}
+  \langle s_a , \Pi^{-1} s_{c;ji}\rangle Q_{cb} -  \langle s_a ,\Pi^{-1}s_{p;i} \rangle \langle s_p,\Pi^{-1}s_{c;j}\rangle Q_{cb} $$
$$+  Q_{ap}\langle s_{a;i} , \Pi^{-1} s_{c;j}\rangle Q_{cb} $$  
$$= - Q_{ac}\langle s_{c;i} , \Pi^{-1} s_p \rangle \langle s_{p;j} , \Pi^{-1} s_b \rangle
-  \langle s_a , \Pi^{-1} s_{c;i}\rangle Q_{cm} \langle s_{m;j} , \Pi^{-1} s_b \rangle
 +Q_{ac}\langle s_{c;ji} , \Pi^{-1} s_b \rangle+Q_{ac}\langle s_{c;j} , \Pi^{-1} s_{b;i} \rangle Q_{pb}$$
$$ -Q_{ac}\langle s_{c;j} ,  \Pi^{-1}s_{p}  \rangle \langle s_{p;i}, \Pi^{-1}  s_b  \rangle
            -  \langle s_a , \Pi^{-1} s_{c;j}\rangle  Q_{cm}\langle s_{m;i} , \Pi^{-1} s_b\rangle                  
 -  \langle s_a , \Pi^{-1} s_{c;j}\rangle  \langle s_c, \Pi^{-1} s_{m;i}\rangle Q_{mb} $$
$$
+  \langle s_a , \Pi^{-1} s_{c;ji}\rangle Q_{cb} -  \langle s_a ,\Pi^{-1}s_{p;i} \rangle \langle s_p,\Pi^{-1}s_{c;j}\rangle Q_{cb} +  Q_{ap}\langle s_{a;i} , \Pi^{-1} s_{c;j}\rangle Q_{cb}.  $$  
The orthogonal projection of the symmetric matrices onto the tangent space at $V$ of the rank $k$ projections is $$A\to VAQ+QAV.$$ 
Hence the projection of $ V_{ab,ji}(x)e_b\otimes e_a^*$ onto the tangent space at $V$ of the rank $k$ projections is 
$$(- Q_{ac}\langle s_{c;i} , \Pi^{-1} s_p \rangle \langle s_{p;j} , \Pi^{-1} s_b \rangle
 +Q_{ac}\langle s_{c;ji} , \Pi^{-1} s_b \rangle$$
$$ -Q_{ac}\langle s_{c;j} ,  \Pi^{-1}s_{p}  \rangle \langle s_{p;i}, \Pi^{-1}  s_b  \rangle           
 -  \langle s_a , \Pi^{-1} s_{c;j}\rangle  \langle s_c, \Pi^{-1} s_{m;i}\rangle Q_{mb} $$
$$
+  \langle s_a , \Pi^{-1} s_{c;ji}\rangle Q_{cb} -  \langle s_a ,\Pi^{-1}s_{p;i} \rangle \langle s_p,\Pi^{-1}s_{c;j}\rangle Q_{cb})e_b\otimes e_a^* .  $$ 
Hence as a map into the Grassmannian, we have at the center of a normal coordinate system
$$d_V^*di_{(h,H)} = ( 2Q_{ac}\langle s_{c;j} , \Pi^{-1} s_p \rangle \langle s_{p;j} , \Pi^{-1} s_b \rangle
 -Q_{ac}\langle s_{c;jj} , \Pi^{-1} s_b \rangle$$
$$
-  \langle s_a , \Pi^{-1} s_{c;jj}\rangle Q_{cb} +  2\langle s_a ,\Pi^{-1}s_{p;j} \rangle \langle s_p,\Pi^{-1}s_{c;j}\rangle Q_{cb})e_b\otimes e_a^* .  $$ 
Here $d_V^*$ denotes the adjoint of the exterior derivative on $f^*TGr(k,N)$ valued forms, and $d_V^*di_{(h,H)}$ is the tension of $i_{(h,H)}$.  

Suppose now that the elements of $H$ satisfy a Bochner identity 
\begin{equation}\label{bochner}\nabla^*\nabla s + Fs = 0,\end{equation}
where $F$ is any algebraic operator. Then 
$$Q_{ac}\langle s_{c;jj} , \Pi^{-1} s_b \rangle = 0 = \langle s_a , \Pi^{-1} s_{c;jj}\rangle Q_{cb},$$ and  we have 
\begin{equation}\label{balanceharmonic}d^*di_{(h,H)} =   2Q_{ac}\langle s_{c;j} , \Pi^{-1} s_p \rangle \langle s_{p;j} , \Pi^{-1} s_b \rangle e_b\otimes e_a^*
  +  2\langle s_a ,\Pi^{-1}s_{p;j} \rangle \langle s_p,\Pi^{-1}s_{c;j}\rangle Q_{cb} e_b\otimes e_a^* . \end{equation}
\begin{proposition}If the subspace $H$ satisfies (\ref{bochner}) with respect to an $(h,H)$ balanced connection,
then $i_{(h,H)}$ is harmonic. 
\end{proposition}
\begin{proof}If the connection is balanced, then $0 = Q_{ac}\langle s_{c;j} , \Pi^{-1} s_p \rangle \langle s_{p;j} , \Pi^{-1} s_b \rangle$, and therefore $d^*di_{(h,H)} = 0.$
\end{proof}

Let $P$ denote the space of $N\times N$ semipositive matrices of rank $k$. Then $P$ is a bundle over $Gr(k,N)$ and $i_{(H,h)}$ has a lift $S_{(H,h)}$ to $P$ given by 
$$S(x) = \langle s_a(x),s_b(x)\rangle e_b\otimes e_a^*.$$
\begin{proposition}If $M$ is Kahler and the subspace $H$ satisfies 
$\db_A^*\db_A + F$, $F$ algebraic, and $\db_A$ balanced,
then $i_{(h,H)}$ is a critical point of 
\begin{equation}\label{defenerg}E_S(f):= \|S^{\frac{1}{2}}\db f\|^2,\end{equation}
where $f$ is varied without varying $S$.  
\end{proposition}
\begin{proof} The complex structure operator on $T_VGr(k,N)$ is given by 
\begin{equation}\label{defJ}JX = i(V-Q)X.\end{equation}
 By $\db i_{(h,H)}(x)$ we denote the component of $di_{(h,H)}(x)$ mapping $T_x^{0,1}M\to T^{1,0}_{i_{(h,H)(x)}}Gr(k,N)$. Hence 
\begin{equation}\label{defdb}\db i_{(h,H)} = Q_{ac}\langle s_{c;\bar j} , \Pi^{-1} s_b \rangle e_b\otimes e_a^*\otimes d\bar z^j.\end{equation}
                                             
 Since  $\db_A$ is balanced, we have 
$$\db^*\db i_{(h,H)} =    Q_{ac}\langle s_{c;\bar j} ,  \Pi^{-1}s_{p}  \rangle \langle s_{p;j}, \Pi^{-1}  s_b  \rangle  e_b\otimes e_a^* .  $$

Suppose now that we alter the metric on the tangent space in an $S-$ dependent manner as follows. 
$$\langle T_1,T_2\rangle_S:= \langle T_1,ST_2\rangle .$$
Then 
$$\frac{\p}{\p x^j}\langle T_1,T_2\rangle_S =\langle \nabla^S_jT_1,T_2\rangle_S + \langle T_1,\nabla^S_{\bar j}T_2\rangle
=  \langle \nabla_jT_1,ST_2\rangle  + \langle T_1,\nabla_{\bar j}(\langle s_a,s_b\rangle e_b\otimes e_a^*T_2)\rangle
$$
$$=  \langle \nabla_jT_1,ST_2\rangle  + \langle T_1, \langle s_a,s_{b,j}\rangle e_b\otimes e_a^*T_2 +S\nabla_{\bar j}T_2 \rangle$$
$$=  \langle (\nabla_j+ \langle  s_{b,j},s_a\rangle e_a\otimes e_b^*)T_1,ST_2\rangle  + \langle T_1, S\nabla_{\bar j}T_2 \rangle.$$
Thus, $\nabla_j^S = \nabla_j+ \langle  s_{b,j},s_a\rangle e_a\otimes e_b^* $ is metric compatible. In particular, 
$i_{(h,H)}$ is formally a critical point of  the energy functional 
$$E_S(f):= \|S^{\frac{1}{2}}\db f\|^2,$$
where $S$ is fixed. 
\end{proof}

\section{The Discrete Dynamics of Balanced Semi-Connections}\label{balance}
It is easy to analyze the pullback via canonical maps of canonical connections, semiconnections, and metrics. 
First we consider the case of semiconnections, as we can alter the semiconnection without varying the metric. 

Given $(h,H,\db_A)$, define a new semiconnection 
\begin{equation}\label{semidef}\db_{T^{0,1}_{(h,H)}A} := \db_A - \db_As_a\langle\cdot,\Pi^{-1}(x,x)s_a\rangle.\end{equation}
From this explicit definition, we see that this semiconnection is actually independent of $A$: 
\begin{equation}\label{semind}\db_{T^{0,1}_{(h,H)}A} =  \db_{T^{0,1}_{(h,H)}B},\forall B.\end{equation}
Its curvature satisfies 
$$F_{T^{0,1}_{(h,H)}A}^{0,2} = F_{A}^{0,2} - F_A^{0,2}s_a\langle\cdot,\Pi^{-1}s_a\rangle  
+  \db_As_a\langle\cdot,\Pi^{-1}\p_A s_a\rangle   
-  \db_As_a\langle \db_A s_b,\Pi^{-1} s_a\rangle \langle\cdot,\Pi^{-1}s_b\rangle  $$
$$
-  \db_As_a\langle s_b,\Pi^{-1}s_a\rangle   \langle\cdot,\Pi^{-1} \p_A s_b\rangle 
+   \db_As_a\langle \db_As_b,\Pi^{-1} s_a\rangle  \langle\cdot,\Pi^{-1} s_b\rangle, $$
which reduces to 
\begin{equation}\label{f02can}F_{T^{0,1}_{(h,H)}A}^{0,2}=  \db_As_a\langle\cdot,\Pi^{-1}\p_A s_b\rangle Q_{ba} . \end{equation}

More generally, we may define 
\begin{equation}\label{fulldef}d_{T_{(h,H)}A} := d_A - d_As_a\langle\cdot,\Pi^{-1} s_a\rangle.\end{equation}
Observe that on complex manifolds 
\begin{equation}\label{semisame}\db_{T_{(h,H)}A}  = \db_{T^{0,1}_{(h,H)}A}.\end{equation}
Then we have 
\begin{equation}\label{canF}F_{T_{(h,H)}A} =   d_As_a\langle\cdot,\Pi^{-1}d_A s_b\rangle Q_{ba} . \end{equation}
This connection is no longer metric compatible; so, we must change $h$ also. We define 
$$T_{(h,H)}h(\cdot,\cdot) := c_Hh(\cdot,\Pi^{-1}\cdot ) = c_H\langle\cdot,\Pi^{-1}\cdot\rangle,$$
where 
$$c_H:= \frac{\text{dim }H}{rank(E)Vol(M)}.$$
The new connection is compatible with respect to this new metric. 

Next we consider modifications of $H$. Once again we consider the complex case and the general case. For complex manifolds we consider the energy functional 
$$\Phi^{0,1}(h,H,\db_A): = \sum_a\|\db_As_a\|^2.$$ For Riemannian manifolds we consider the functional 
$$\Phi (h,H,d_A): = \sum_a\|d_As_a\|^2.$$
Let $e^{0,1}$ and $e$ denote the corresponding local energy densities: 
$$e^{0,1}(h,H,\db_A)(x) := \sum_a|\db_A s_a|^2(x), $$ and 
$$e(h,H,d_A)(x) := \sum_a|d_A s_a|^2(x).$$
Set $e^{1,0} = e-e^{0,1}.$  
At each $x\in M$, we may decompose $H = G_x\oplus P_x$, where $G_x$ is the kernel of the evaluation at $x$ map, and $P_x$ is the $L_2$ orthogonal complement of $G_x$. Decompose the local energy density into 
$$e^{0,1}(h,H,\db_A)(x) = e^{0,1,G}(h,H,\db_A)(x) + e^{0,1,P}(h,H,\db_A),$$
where 
$$e^{0,1,G}(h,H,\db_A)(x):= \sum_i|\db_A u_i|^2(x) = \sum_{c,m}Q_{cm}\langle s_{c;\bar j},s_{m;\bar j}\rangle (x), $$ $e^{0,1,P}(h,H,\db_A)(x):= \sum_{j=1}^{k}|\db_A v_j|^2(x), $ with $\{u_i\}_i$ and $\{v_j\}_j$ $L_2$ unitary bases of $G_x$ and $P_x$ respectively. Let 
$$\Phi^{0,1}_G(h,H,\db_A) := \int e^{0,1,G}(h,H,\db_A)dv  = E_S(i_{(h,H)}),$$
with $E_S$ defined in (\ref{defenerg}). Observe that $\Phi^{0,1}_G(h,H,\db_A)$ is independent of the choice of semiconnection. Similarly define 
$ \Phi^{1,0}_G(h,H,\p_A)$ and $\Phi_G(h,H,d_A)$.  Then we have    
$$ e^{0,1,G}(h,H,\db_A)(x) =  e^{0,1}(h,H,\db_{T^{0,1}_{(h,H)}A})(x).$$
So, $\Phi^{0,1}(h,H,\db_A)$ is nonincreasing when we send $\db_A \to \db_{T^{0,1}_{(h,H)}A}$. It is strictly decreasing unless the semiconnection is balanced. 
Next we replace $H$ with a {\em choice} of $T_{(h,\db_B)}H$ determined as follows.  
Let $ T_{(h,\db_B)}H$ be a subspace $W$ of $C^\infty(M,E)$ on which $\Phi^{0,1}(h,W,\db_B)$ is minimal. Generically such minimizing $W$ are unique. The existence of $W$ is a linear eigenvalue problem, elementary on compact manifolds. 

Thus we have a discrete dynamical system minimizing $\Phi^{0,1}$. 
Consider the map 
\begin{equation}\label{gen01}T^{0,1}:(h,H,\db_A)\to (h,T_{(h,\db_{T^{0,1}_{(h,H)}A})}H,\db_{T^{0,1}_{(h,H)}A}).
\end{equation}
Then 
$$\Phi^{0,1}(T^{0,1}(h,H,\db_A))\leq \Phi^{0,1}(h,H,\db_A),$$
with equality only if $\db_A$ is $(h,H)$ balanced, $H$ minimizes $\Phi^{0,1}(h,\cdot,\db_A)$ and is therefore spanned by eigensections of $\db_A^*\db_A$. Observe the discrete dynamical system generated by $T^{0,1}$ terminates if $\Pi =:\Pi_H$ is not invertible. If we replace $E$ by $E\otimes L^p$ for $p$ sufficiently large depending on $(h,\db_A)$, then Bergman kernel asymptotics for non holomorphic bundles (see \cite{CS} and\cite{MM}) can be used to show there is a unique $H$ minimizing $\Phi^{0,1}(h,\cdot,\db_A)$, and this $H$ is admissible. 

We may define a similar dynamical system in the Riemannian case, but now we must modify the metric at each step also to keep the connection metric compatible. Altering the metric disturbs the monotonicty of the energy under the natural transformation 
\begin{equation}\label{gen}T :(h,H,d_A)\to (h,T_{(h,d_{T_{(h,H)}A})}H,d_{T_{(h,H)}A}),\end{equation}
where $ T_{(h,d_B)}H$ is a subspace $W$ of $C^\infty(M,E)$ on which $\Phi(h,W,d_B)$ is minimal.

\section{Obstructions}\label{obstructions}

\begin{definition}Fix a metric $h$. 
We call a sequence of admissible subspaces $\{H_n\}_n$ of $C^\infty(M,E)$  {\em stable} if there exists $C>0$ such that 
$|\Pi_{H_n}^{-1}|\leq C,\,\,\forall n.$ We call a sequence stable on the submanifold $Z$ if there exists $C>0$ such that 
 $|\Pi_{H_n}^{-1}(x,x)|\leq C,\,\,\forall n, \forall x\in Z.$
\end{definition}
We have the following proposition. 
\begin{theorem}\label{ch1e} 
Suppose that $ch_1(E)\not\in H^{1,1}(M)$. Let $\{H_n\}_n $ be a stable sequence of subspaces and $\{A_n\}_n $
 be a sequence of smooth hermitian connections on $E$. If $\int_Me^{0,1}(h,H_n,\db_{A_n})dv\rightarrow  0,$ then 
$\int_Me (h,H_n,d_{B_n})dv \rightarrow  \infty,$ for {\em any} sequence of hermitian connections $\{B_n\}_n$.\\
Let $[Z]\in H_2(M,\IZ)$ have nonzero pairing with the $(0,2)\text{ component of }c_1(E).$
Let $\{H_n\}_n $ be a sequence of admissible subspaces stable on $Z$ and $\{A_n\}_n $
 a sequence of smooth connections. If $\int_Ze^{0,1}(h,H_n,\db_{A_n})dv_Z\rightarrow  0,$ then 
$\int_Ze (h,H_n,d_{B_n})dv_Z \rightarrow  \infty,$ for {\em any} sequence of hermitian connections $\{B_n\}_n$. 
\end{theorem}
\begin{proof}
We observe that 
From (\ref{f02can}) we have $$F_{T^{0,1}_{(h,H_n)}A_n}^{0,2} = \db_{A_n}s_a\langle\cdot,\Pi_{H_n}^{-1}\p_{A_n} s_b\rangle Q_{ba}.$$
Hence we see that 
\begin{equation}\label{ineq10}|F_{T^{0,1}_{(h,H_n)}A_n}^{0,2}(x)|^2\leq  C e^{0,1} (h,H_n,\db_{A_n})(x)e (h,H_n,d_{A_n})(x) .\end{equation}
Let $h$ be a harmonic representative of $ch_1(E)$. Then 
\begin{equation}\label{lwrbnd}0\not = \|h^{0,2}\|^2\leq \|\tr F_{T^{0,1}_{(h,H_n)}A_n}^{0,2}\|^2_{L_2}\leq  C \Phi^{0,1} (h,H_n,\db_{A_n}) \Phi (h,H_n,d_{A_n}).\end{equation}
Hence $\Phi^{0,1} (h,H_n,\db_{A_n})\to 0$ implies $\Phi (h,H_n,d_{A_n})\to\infty$. 
Now let $B_n$ be any connection, and $\{s_a\}_a$ a unitary $L_2$ basis for $H_n$. Then 
$$\sum_a |\p_{B_n}s_a |^2(x)\geq \sum_a |\p_{T_{(h,H_n)}A_n}s_a |^2(x)\geq c_n|\Pi^{-1/2}_{H_n}\p_{T_{(h,H_n)}A_n}s_a |^2(x),$$
and 
$$\sum_a |\db_{A_n}s_a |^2(x)\geq \sum_a |\db_{T^{0,1}_{(h,H_n)}A_n}s_a |^2(x)\geq c_n\sum_a |\Pi^{-1/2}_{H_n}\db_{T^{0,1}_{(h,H_n)}A_n}s_a |^2(x),$$ 
where $\{\ln(c_n)\}_n$ is a bounded sequence, by the stability hypothesis. Setting $h_n(\cdot,\cdot) = h(\cdot,\Pi_{H_n}^{-1}\cdot\rangle$ and applying (\ref{ineq10}) to the triple $(h_n,H_n,T_{(h,H_n)}A_n)$ gives 
\begin{multline}\label{ineq12}|F_{T^{0,1}_{(h_n,H_n)}T_{(h,H_n)}A_n}^{0,2}(x)|^2\leq  C e^{0,1} (h_n,H_n,\db_{T_{(h,H_n)}A_n})(x)e (h_n,H_n,T_{(h,H_n)}d_{A_n})(x)\\\leq Cc_n^{-2} e^{0,1} (h,H_n,\db_{A_n})(x)e (h,H_n, d_{B_n})(x) .\end{multline}
We may now use this inequality to replace   $\Phi (h,H_n,d_{A_n})$ by $\Phi (h,H_n,d_{B_n})$ in (\ref{lwrbnd}) to prove $\Phi (h,H_n,d_{B_n})\to\infty$ as claimed. 
The proof of the second result follows from  localizing the preceding argument to $Z$. 
\end{proof}
Let $F^{\cdot}$ denote the Hodge filtration. Taking higher exterior powers of $F^{0,2}$ and applying the preceding computation yields the following theorem. 
\begin{theorem} \label{chpe}
Suppose that $ch_p(E)\not\in F^1H^{2p}(M)\cap \bar F^1H^{2p}$. Let $\{H_n\}_n $ be a stable sequence of subspaces and $\{A_n\}_n $
 be a sequence of smooth hermitian connections. Let $\frac{1}{q}+\frac{1}{q'} =1,$ $q\geq 1$. If $e^{0,1}(A_n,H_n)\stackrel{L_{pq}}{\rightarrow} 0,$ then 
$e^{1,0}(B_n,H_n)\stackrel{L_{pq'}}{\rightarrow} \infty,$ for any sequence of hermitian connections $\{B_n\}_n$.\\
Let $[Z]\in H_{2p}(M,\IZ)$ have nonzero pairing with the $(0,2p)\text{ component of }ch_p(E).$
Let $\{H_n\}_n $ be a sequence of admissible subspaces stable on $Z$ and $\{A_n\}_n $
 a sequence of smooth connections. If $e^{0,1}(A_n,H_n)\stackrel{L_{pq}(Z)}{\rightarrow} 0,$ then 
$e^{1,0}(B_n,H_n)\stackrel{L_{pq'}(Z)}{\rightarrow} \infty,$ for any sequence of hermitian connections $\{B_n\}_n$, 
\end{theorem}
To go deeper into the Hodge diamond, we need bounds on $F_A^{1,1}.$ Using the stability assumption, we can allow our metrics to vary also. Then from (\ref{canF}) we have 

\begin{equation}\label{canF11}F_{T(h,H)A}^{1,1} =  [\db_{A}s_a \langle\cdot,\Pi^{-1}\db_As_b\rangle +\p_{A}s_a \langle\cdot,\Pi^{-1}\p_As_b\rangle]Q_{ba} .\end{equation}
Let $$\phi_B\omega := \frac{1}{m}e^*(\omega)F_B^{1,1},$$ where $\omega$ denotes the Kahler form. Then 
\begin{equation}\label{trf}m\phi_{T(h,H)A}   =   i[ \langle \db_As_a,\Pi^{-1}\db_As_{b,}\rangle 
-   \langle \p_As_{a},\Pi^{-1} \p_As_{b}\rangle]Q_{ba} .\end{equation}
Hence
$$\deg(E) := \frac{i}{2\pi}\int_M\omega^{m-1}\wedge \tr F_{T(h,H)A}$$
$$ = \frac{(m-1)!}{2\pi}\int_M(\langle \p_As_{a},\Pi^{-1} \p_As_{b}\rangle -  \langle \db_As_a,\Pi_d^{-1}\db_As_{b,}\rangle )Q_{ba}dv .$$
So, we see that a stable sequence with bounded $\Phi_{G}^{0,1}$ must also have bounded $\Phi_G^{1,0},$ and therefore bounded $\Phi_G$. Similar results hold for energy densities integrated over complex submanifolds, where the energy involves only derivatives in directions tangent to the submanifold.     
\begin{theorem}\label{thm0}
If $c_1(E)\not\in F^1H^{2}\cap \bar F^1H^{2}$, the energy, $\Phi^{0,1}$ remains bounded below on stable sequences. 
\end{theorem}
\begin{proof}If $\Phi^{0,1}(h_n,H_n,\db_{A_n})\to 0$ then $\Phi_G^{0,1}(h_n,H_n,\db_{A_n})\to 0$. Hence $\Phi_G(h_n,H_n,d_{A_n})\to 0$. The stability assumption implies then that $\Phi_G(h_n,H_n,d_{T_{h_n,H_n}A_n})\to 0$. This contradicts Theorem \ref{ch1e}. 
\end{proof}

 In $\tr F_A^p$, the only terms with no $\db s_a$ factors are in $\tr (F_A^{1,1})^p$. 
 So using Holder's inequality and choosing exponents optimally we come to the following theorem deeper in the Hodge diamond.  
\begin{theorem}\label{thm1} 
Suppose that $ch_p(E)\not\in F^{p-j+1}H^{2p}(M)\cap \bar F^{p-j+1}H^{2p}$. Let $\{H_n\}_n $ be a stable sequence of subspaces and $\{A_n\}_n $
 be a sequence of smooth hermitian connections. If $e^{0,1}(A_n,H_n)\stackrel{L_{p}}{\rightarrow} 0,$ then 
$e^{1,0}(B_n,H_n)\stackrel{L_{p}}{\rightarrow} \infty,$ for any sequence of hermitian connections $\{B_n\}_n$. 
If $e^{0,1}(A_n,H_n)\stackrel{L_{\infty}}{\rightarrow} 0,$ then 
$e^{1,0}(B_n,H_n)\stackrel{L_{p-\frac{j}{2}}}{\rightarrow} \infty,$ for any sequence of hermitian connections $\{B_n\}_n$.\\
Let $[Z]\in H^{2m-2p}(M,\IZ)$ have nonzero pairing with the $(0,p)\text{ component of }ch_p(E).$
Let $\{H_n\}_n $ be a sequence of admissible subspaces stable on $Z$ and $\{A_n\}_n $
 a sequence of smooth connections. If $e^{0,1}(A_n,V_n)\stackrel{L_{p}(Z)}{\rightarrow} 0,$ then 
$e^{1,0}(B_n,H_n)\stackrel{L_{p}(Z)}{\rightarrow} \infty,$ for any sequence of hermitian connections $\{B_n\}_n$. 
If $e^{0,1}(A_n,H_n)\stackrel{L_{\infty}(Z)}{\rightarrow} 0,$ then 
$e^{1,0}(B_n,H_n)\stackrel{L_{p-\frac{j}{2}}(Z)}{\rightarrow} \infty,$ for any sequence of hermitian connections $\{B_n\}_n$. 
\end{theorem}
\begin{proof}Let $h$ be the harmonic representative of $\tr F_{ T_{(h,H_n)}A_n}^p$. Then the assumption on the location of $ch_p(E)$ in the Hodge filtration implies that the component $h^{p-i,p+i}$ of $h$ of bidegree $(p-i,p+i)$ is nonzero for some $i\geq j$. Then 
$$\|h^{p-i,p+i}\|_{L_2}^2 =  \langle h^{p-i,p+i},\tr F_{ T_{(h,H_n)}A_n}^p\rangle_{L_2}\leq \|h^{p-i,p+i}\|_{L_2}\|(\tr F_{ T_{(h,H_n)}A_n}^p)^{p-i,p+i}\|_{L_2}.$$
Hence 
$$\|h^{p-i,p+i}\|_{L_2} \leq  \|(\tr F_{ T_{(h,H_n)}A_n}^p)^{p-i,p+i}\|_{L_2}\leq c\sum_{a+b=p, a\geq i}\||F^{0,2}_{T^{0,1}_{(h,H_n)}A_n}|^a|F_{T_{(h,H_n)}A_n}^{1,1}|^b\|_{L_2}$$
$$\leq c_n \sum_{i\leq a\leq p-i}\sqrt{\int e^{0,1}(h,H_n,\db_{A_n})^{p-a}e^{1,0}(h,H_n,\p_{A_n})^{p+a}dv}.$$
Now apply Holder's inequality to deduce the desired results, with $B_n=A_n$. Then, arguing as in Theorem \ref{ch1e}, the result for  $A_n = T_{(h,H_n)A_n}$ implies the result for arbitrary $B_n$. 
\end{proof}

\section{Gauge Transformations}
Let $u$ be a unitary gauge transformation. Then 
$$\Phi^{0,1}(h,H,\db_A) = \Phi^{0,1}(h,u^{-1}H,u^{-1}\db_Au).$$
If $U$ is a complex gauge transformation, then 
$$\Phi^{0,1}(h,H,\db_A) = \Phi^{0,1}(h(U\cdot,U\cdot),U^{-1}H,U^{-1}\db_AU).$$
 
\begin{definition}We call a triple $(h,H,\db_A)$ a {\em real gauge soliton} if 
$$T^{0,1}(h,H,\db_A) = (h,u^{-1}H,u^{-1}\db_Au),$$
for some unitary gauge transformation $u$. \\
We call a triple $(h,H,\db_A)$ a {\em complex gauge soliton} if 
$$T^{0,1}(h,H,\db_A) = (h(U\cdot,U\cdot),U^{-1}H,U^{-1}\db_AU),$$
for some complex gauge transformation $U$. 
\end{definition}
Hence solitons are fixed points of the dynamical system, modulo gauge equivalence. 
We observe that 
$$\db_{T^{0,1}(h,H)A} = \db_A - [\db_A,\Pi(x,x)]\Pi^{-1}(x,x) + s_a\langle ,\Pi^{-1}\p_A s_A\rangle.$$
Hence $\db_{T^{0,1}(h,H)A}$ is complex gauge equivalent to 
$  \db_A    +  \Pi^{-1}s_a\langle ,\p_A s_A\rangle .$
So, we see that a sufficient condition for the sequence $\{A_n\}_n$ to be a complex gauge soliton is that the  
$$\Phi^{1,0}(h,H,\p_A ) =\Phi^{1,0}_G(h,H,\p_A ).$$
Solitons never take us closer to holomorphicity since they are evolution by complex gauge transformations.

If we wish to speed convergence of the dynamical system, it is natural to add a unitary gauge transformation at each step.  Let $u_n$ be unitary endomorphism of the bundle which minimizes 
$$\|u_ndb_{A_{n+1}}u_n^{-1}-\db_{A_n }\|^2 = \|db_{A_{n+1}}-u_n^{-1}\db_{A_n }u_n\|^2= \|(A_{n+1}-A_n)-u_n^{-1}[\db_{A_n},u_n] \|^2.$$

\section{Singular Canonical Connections}
When $H$ does not globally generate $E$, it does not define a canonical map to a Grassmannian, and the discrete dynamical systems generated by $T$ and $T^{0,1}$ terminate. In this case, we can still construct an approximate canonical connection as follows. Let 
$$\Pi_\epsilon: = \Pi+\epsilon I.$$
Define 
the semiconnection 
$$\db_{T_\epsilon^{0,1}(h,H)A} := \db_A-\db_As_a\langle\cdot,\Pi_\epsilon^{-1}s_a\rangle.$$
Then 
$$F_{T_\epsilon^{0,1}(h,H)A}^{0,2}  = F_A^{0,2}   -F_A^{0,2}s_a\langle\cdot,\Pi_\epsilon^{-1}s_a\rangle 
+\db_As_a\langle\cdot,\Pi_\epsilon^{-1}\p_As_b\rangle (\delta_{ab}- \langle s_b,\Pi_{\epsilon}^{-1}s_a\rangle)$$
$$= F_A^{0,2}\epsilon\Pi_\epsilon^{-1} +\db_As_a\langle\cdot,\Pi_\epsilon^{-1}\p_As_b\rangle (\delta_{ab}- \langle s_b,\Pi_{\epsilon}^{-1}s_a\rangle).$$
Similarly define the connection
$$d_{T_\epsilon(h,H)A} := d_A-d_As_a\langle\cdot,\Pi_\epsilon^{-1}s_a\rangle .$$
This connection is hermitian with respect to the metric 
$$\langle\cdot, \cdot\rangle_\epsilon:=\langle\cdot,\Pi_\epsilon^{-1}\cdot\rangle.$$
Therefore $\Pi_\epsilon^{-1/2}d_{T_\epsilon(h,H)A}\Pi_\epsilon^{1/2}$ is hermitian with respect to the orginal metric. 
We have 
$$F_{T_\epsilon(h,H)A}  =  F_A\epsilon\Pi_\epsilon^{-1} +d_As_a\langle\cdot,\Pi_\epsilon^{-1}d_As_b\rangle (\delta_{ab}- \langle s_b,\Pi_{\epsilon}^{-1}s_a\rangle),$$
and 
\begin{equation}\label{whatsup}\Pi_\epsilon^{-1/2}F_{T_\epsilon(h,H)A}\Pi_\epsilon^{ 1/2}   = \Pi_\epsilon^{-1/2}F_A\epsilon\Pi_\epsilon^{-1/2} +\Pi_\epsilon^{-1/2}d_As_a\langle\cdot,\Pi_\epsilon^{-1/2}d_As_b\rangle (\delta_{ab}- \langle s_b,\Pi_{\epsilon}^{-1}s_a\rangle).\end{equation}
Let $K(x)$ denote the orthogonal projection onto the kernel of $s_a(x)\langle\cdot, s_a(x)\rangle.$ Taking the limit as $\epsilon\to 0$, 
\begin{equation}\label{doc}\Pi_\epsilon^{-1/2}F_A\epsilon\Pi_\epsilon^{-1/2}\to KF_AK.\end{equation}

These curvature formulas can now be used to find partial extensions of Theorems \ref{thm0} and \ref{thm1} to unstable sequences, but we will not pursue that application here.  They also can be used to gain insight into mass gap questions. For example, if $A_n$ is a sequence of irreducible $su(2)$ Yang Mills minimizing connections and the first eigenvalue  of $d_{A_n}^*d_{A_n}$ approaches zero, then the $L_4$ norm of $d_{A_n}s_n$, with $s_n$ a normalized first eigensection must 
blowup; otherwise, the equation (\ref{whatsup}) with $H=\langle s_n\rangle$ contradicts minimality of $\|F_{A_n}\|^2.$

\section{Entropy}
In this section we consider canonical maps generated by heat kernels. 
As is perhaps known to experts (see \cite{MM},\cite{T},\cite{Z}), the canonical connections induced by these maps approximate the original connection as $t\to 0$. This computation  is not motivated by our Hodge theoretic questions; it is simply an amusing application of the expressions for the curvature of connections induced by maps to Grassmannians.

Given a metric $h$, a connection $d_A$, and an endomorphism $r$ of $E$, let $L= d_A^*d_A+r$. Let $\{\psi_j\}_{j=1}^\infty$ be an $L_2$ orthonormal basis of eigensections of $L$, $L\psi_j = \lambda_j$. Let $H_{N}$ denote the span of the eigenvectors of $L$ with  eigenvalue $\leq N$. Instead of considering the orthogonal projection onto $H_N$, we set 
$$\Pi_{N,t}:= \sum_{\lambda_j\leq N}e^{-t\lambda_j}\psi_j\langle\cdot,\psi_j\rangle.$$
This is the composition of the projection onto $H_N$ and $e^{-tL}$. Then 
$$\lim_{N\to\infty} \Pi_{N,t} = k_t,$$
where $k_t$ denotes the Schwartz kernel for $e^{-tL}$. 
Standard heat equation asymptotics give 
$$k_t(x,x) =(4\pi t)^{-n/2}I + O(t^{1-n/2}),$$
with explicitly computable lower order terms. 
Hence for $N = N(t)$ sufficently large, $\Pi_{N,t}(x,x)$ is invertible with 
$$\Pi_{N,t}^{-1}(x,x) = (4\pi t)^{n/2}I + O(t^{1+ n/2}).$$
Then we have the corresponding Grassmann embedding defined by 
$$V_{N,t} = V_{ab,N,t}e_b\otimes e_a^* = 
\langle e^{-t\lambda_a/2}\psi_a,\Pi_{N,t}^{-1}e^{-t\lambda_b/2}\psi_b\rangle e_b\otimes e_a^*.$$
The pullback of the canonical connection is then given by  
$$d_{A,N,t}:= d_A - \sum_{\lambda_j\leq N}e^{-t\lambda_j}d_A\psi_j\langle \cdot, \Pi_{N,t}^{-1}\psi_j\rangle,$$ 
and has curvature 
$$F_{A,N,t} =  \sum_{\lambda_a,\lambda_b\leq N}Q_{ab,N,t}e^{-t\lambda_a/2}d_A\psi_a\langle \cdot, \Pi_{N,t}^{-1}e^{-t\lambda_b/2}d_A\psi_b\rangle,$$
with $Q_{N,t} = I-V_{N,t}. $
On the other hand 
$$\sum_{\lambda_j }e^{-t\lambda_j}d_A\psi_j\langle \cdot, \Pi_{N,t}^{-1}\psi_j\rangle = d_Ak_t(x,y)_{|x=y}\circ \Pi_{N,t}^{-1}(x,x),$$
and from heat equation asymptotics
\begin{equation}\label{lowerorder}|d_Ak_t(x,y)_{|x=y}\circ \Pi_{N,t}^{-1}(x,x)|_{C^1}= O(t).\end{equation}
Hence we see that for $N$ sufficiently large relative to $t^{-1}$ (or $N=\infty$), we have 
$$F_{A} =  \sum_{\lambda_a,\lambda_b\leq N}Q_{ab,N,t}e^{-t\lambda_a/2}d_A\psi_a\langle \cdot, \Pi_{N,t}^{-1}e^{-t\lambda_b/2}d_A\psi_b\rangle + O(t)   $$
$$=  (d_Ak_t(x,y)d_A^*)_{x=y}\Pi_{N,t}^{-1}(x,x) - d_Ak_t(x,y)_{x=y}\Pi_{N,t}^{-1}(x,x)(k_t(y,x)d_A^*)_{y=x}\Pi_{N,t}^{-1}(x,x) + O(t).  $$
Subtracting off the $d_Ak_t(x,y)_{x=y}\Pi_{N,t}^{-1} (k_t(y,x)d_A^*)_{y=x}\Pi_{N,t}^{-1} $ term in the preceding expression induces the projection $Q$. By (\ref{lowerorder}) this term is $O(t^2)$. Hence 
\begin{equation}\label{heatasymp}F_A =  (d_Ak_t(x,y)d_A^*)_{x=y}\Pi_{N,t}^{-1}(x,x)  + O(t) . \end{equation}
This gives the following theorem.
\begin{theorem}\label{entropy}
\begin{equation}\label{entreq}F_{A } =  (4\pi t)^{n/2}\sum_{\lambda_a}e^{-t\lambda_a }d_A\psi_a\langle \cdot,   d_A\psi_a\rangle + O(t) .  \end{equation}
\end{theorem}

Formally, we see in (\ref{heatasymp}) that $F_A$ scales like an entropy - more accurately a thermodynamic energy- the trace of the time derivative of $e^{-tL}$ divided by the trace of $e^{-tL}$.

%%% Bibliography
%
%\bibliographystyle{kw}
 \newcommand{\etalchar}[1]{$^{#1}$}
\def\polhk#1{\setbox0=\hbox{#1}{\ooalign{\hidewidth
  \lower1.5ex\hbox{`}\hidewidth\crcr\unhbox0}}}
\providecommand{\bysame}{\leavevmode\hbox to3em{\hrulefill}\thinspace}

%\bibliography{kw} 

\begin{thebibliography}{CCN{\etalchar{+}}85}

 \bibitem[CS]{CS}{B.~Charbonneau and M.~Stern},
{\em Asymptotic Hodge Theory of Vector Bundles}, {Communications in Analysis and Geometry}, {to appear} 
 	arXiv:1111.0591 [math.DG].

\bibitem[D1]{D1}{S.~Donaldson},
{\em Geometry in Oxford c. 1980–85, Sir Michael Atiyah: a great mathematician
of the twentieth century.} {Asian J. Math.}, {\bf 3} {no. 1} (1999), xliii–xlvii.

\bibitem[D]{D}{S.~Donaldson},
{\em Anti self-dual Yang--Mills connections over complex algebraic surfaces and stable vector bundles}, {Proc. London Math. Soc}, {\bf 50} {no. 3} (1985), 1--26.
  
\bibitem[K]{K}
{J.~Keller},
{{\em Vortex Equations and Canonical Metrics}},
{Math. Ann.} 
{\bf 337} {(2007)},  923--979. 

\bibitem[MM]{MM}
{X.~Ma and G.~Marinescu},
\emph{Holomorphic Morse inequalities and Bergman Kernels},
Birkhauser, Basel, 2007.
  

\bibitem[T]{T}
{G.~Tian},
{{On a set of polarized K¨ahler metrics on algebraic manifolds}},
{J. Differential Geom.} 
{\bf 32} {(1990)}, {no.~1}, 99--130. 

\bibitem[UY]{UY}
{K.~Uhlenbeck and S.~T.~Yau},
{\em On the Existence of Hermitian-Yang-Mills
Connections in Stable Vector Bundles},
{Comm. Pure and Appl. Math.} {\bf 39} {(1986)}, {257--293}, MR0861491, Zbl 0615.58045. 
   
\bibitem[W]{w}
{X.~Wang},
{{\em Balance point and stability of vector bundles over a projective manifold}},
{Math. Research Letters} 
{\bf 9} {(2002)}, {no.~2-3}, 393–411.

\bibitem[Z]{Z}
{S.~ Zelditch},
\emph{ Szego kernels and a theorem of Tian},
{Internat. Math. Res. Notices,} {\bf 6} {(1998)}, {317--331}. 

\end{thebibliography}
%
\end{document}